\documentclass[12pt]{article}
\usepackage{amsmath}
\usepackage{amsfonts}
\usepackage{amssymb}
\usepackage{amsthm}
\usepackage{mathrsfs}
\begin{document}
\title{ Integral Representations in Weighted Bergman Spaces on the Tube Domains }
\author{Yun Huang, Guan-Tie Deng, Tao Qian}

\author{Yun Huang \thanks{Department of Mathematics, Faculty of Science and Technology, University of Macau, Macao (Via Hong
Kong). Email: yhuang12@126.com.},\
Guan-Tie~Deng \thanks
{Corresponding author. School of  Mathematical Sciences, Beijing
Normal University, Beijing, 100875. Email:denggt@bnu.edu.cn.\ This
work was partially supported by NSFC (Grant  11971042) and by SRFDP
(Grant 20100003110004)},\
Tao~Qian \thanks{ Institute of Systems Engineering, Macau University of Science and Technology, Avenida Wai Long, Taipa, Macau). Email: tqian1958@gmail.com. The work is supported by the Macau Science and Technology foundation No.FDCT079/2016/A2 , FDCT0123/2018/A3, and the Multi-Year Research Grants of the University of Macau No. MYRG2018-00168-FST.}
}
\date{}
\maketitle
\begin{center}
\begin{minipage}{120mm}
%\begin{center}{\bf Abstract}\end{center}
{ Herein, the Laplace transform representations for functions of weighted holomorphic Bergman spaces on the tube domains are developed. Then a weighted version of the edge-of-the-wedge theorem is derived as a byproduct of the main results. }\\


\end{minipage}
\end{center}

\section{Introduction}
\newtheorem{theorem}{\sc{Theorem}}
\newtheorem{lemma}{\sc{Lemma}}
\newtheorem{corollary}{\sc{Corollary}}

The classical Paley--Wiener theorem asserts that functions of the classical Hardy space $H^2(\mathbb C^+)$ can be written as the Laplace transforms of $L^2$ functions supported in the right half of the real axis, see \cite{SW}. This theorem has been extended to more general Hardy spaces, including the $H^p$ spaces cases ($0<p\leq\infty$), higher dimensional cases and weighted spaces, see \cite{QXYYY,DQ,DHQ,LI,Li,Q1}. Integral representation theorems have also been investigated for Bergman spaces.

We first introduce some notations and definitions. Let $B$ be a domain (open and connected set) in $\mathbb R^n$ and $T_B=\mathbb R^n+iB\subset \mathbb C^n$ be the tube over $B$. For any element $z=(z_1,z_2,\ldots,z_n)$, $z_k=x_k+iy_k$, by definition, $z\in T_B$ is and only if $x=(x_1,\ldots,x_n)\in\mathbb R^n$ and $y=(y_1,\ldots,y_n)\in B$. The inner product of $z,w\in\mathbb C^n $ is defined as $z\cdot w=z_1w_1+z_2w_2+\cdots+z_nw_n$. The associated Euclidean norm of $z$ is $|z|=\sqrt{z\cdot \bar z}$, where $\bar z=(\bar z_1,\bar z_2,\ldots,\bar z_n)$.

A nonempty subset $\Gamma\subset\mathbb{R}^n$ is called an open cone if it satisfies (i) $0\notin\Gamma$, and (ii) $\alpha x+\beta y \in \Gamma$ for any $x, y\in \Gamma$ and $\alpha,\beta>0$. The dual cone of $\Gamma$ is defined as
$
\Gamma^*=\{y\in\mathbb{R}^n: y\cdot x \geq 0, \ {\rm for\  any }\  x\in \Gamma\}
$, which is clearly a closed convex cone with vertex at $0$. We say that the cone $\Gamma$ is regular if the interior of its dual cone $\Gamma^* $ is nonempty.

For $\frac1p+\frac1q=1$, define
\begin{equation*}
B^p(T_B)=\left\{F:F\mbox{ is holomorphic in }T_B\mbox{ and satisfies }\int_B\left(\int_{\mathbb R^n}|F(x+iy)|^pdx\right)^{q-1}dy<\infty\right\}.
\end{equation*}
Among the previous studies, Genchev showed that the function spaces $B^p$($1\leq p\leq2$), in the one- and multi-dimensions in \cite{Gen1} and \cite{Gen2}, respectively, admit integral representations in the Laplace transform form. These results can be applied to the Bergman spaces
\begin{equation*}
A^p(T_{\Gamma})=\left\{F:F\mbox{ is holomorphic on }T_{\Gamma} \mbox{ and satisfies }\int_{T_{\Gamma}}|F(x+iy)|^pdxdy<\infty\right\}
\end{equation*}
to obtain the corresponding integral representation results for $A^p(T_{\Gamma})$ in the range $1\leq p\leq 2$(\cite{Gen5}).
\\

In this paper we initiate a study on a class of function spaces, denoted by $A^{p,s}(B,\psi)$, of which each is associated with a weight function of the form $e^{-2\pi\psi(y)}$, where $\psi(y)\in C(B)$ is continuous on  $B$. The space $A^{p,s}(B,\psi)$($0<p\leq\infty,0<s\leq\infty$) is the collection of functions $F(z)$ that are holomorphic in $T_B$ and satisfy
\begin{equation*}
\|F\|_{A^{p,s}(B,\psi)}=\left(\int_{B}\left(\int_{\mathbb R^n}|F(x+iy)e^{-2\pi\psi(y)}|^pdx\right)^sdy\right)^{\frac{1}{sp}}<\infty \mbox{, }0<p,s<\infty,
\end{equation*}
\begin{equation*}
\|F\|_{A^{p,\infty}(B,\psi)}=\sup\left\{e^{-2\pi\psi(y)}\left(\int_{\mathbb R^n}|F(x+iy)|^pdx\right)^{\frac1p},y\in B\right\}<\infty\mbox{ , }0<p<\infty, s=\infty
\end{equation*}
and
\begin{equation*}
\|F\|_{A^{\infty,\infty}(B,\psi)}=\sup\left\{e^{-2\pi\psi(y)}|F(x+iy)|,x\in\mathbb R^n,y\in B\right\}<\infty\mbox{, }p=\infty,s=\infty.
\end{equation*}

This paper is structured as follows. In \S2, we introduce our main work on the integral representation for $A^{p,s}(B,\psi)$, which is separated into three cases, namely, $A^{p,s}(B,\psi)$ for $1\leq p\leq2$, $A^{p,s}(B,\psi)$ for $0<p<1$ and $A^{p,s}(\Gamma,\psi)$ for $p>2$, corresponding to Theorem 1, 2 and 3 respectively. The proofs are given in \S3. Finally, some results, referring to Corollary 2, Theorem 4 and Theorem 5, are derived as applications of the integral representation theorems claimed in \S2.

\section{Main results}

In order to introduce our main results, we define the set
\begin{equation}
U_{\alpha}(B,\psi)=\left\{t\in\mathbb R^n:\int_{B}e^{-2\pi \alpha(t\cdot y+\psi(y))}dy<\infty\right\} \label{support}
\end{equation}
 for $ \alpha\in (0,\infty) $ and
\begin{equation}
  U_{\infty}(B,\psi)=\{t: \inf_{y\in\Gamma}(y\cdot t+\psi(y))>-\infty\} \label{support-2}
\end{equation}
for $\alpha=\infty$.\\

The representation theorem for $A^{p,s}(B,\psi)$, where $1\leq p\leq2$ and $0<s\leq\infty$, is stated as follows.
\begin{theorem}
Assume that $1\leq p\leq2$, $0<s\leq\infty$, then each $F(z)\in A^{p,s}(B,\psi)$ admits an integral representation in the form
\begin{equation}
F(z)=\int_{\mathbb{R}^n} f(t)e^{ 2\pi i t\cdot z}dt,\ z\in T_{B}, \label{mainconclusion}
\end{equation}
in which, for $p=1$, $f(t)\in C(\mathbb R^n)$ satisfies
\begin{equation}
|f(t)|\left(\int_{B}e^{-2s\pi (y\cdot t+\psi(y))}dy\right)^{\frac1s}\leq \|F\|_{A^{1,s}(B,\psi)} \label{g0lq}
\end{equation}
and, for $1<p\leq2$ and $\frac1p+\frac1q=1$, $f(t)$ is a measurable function that satisfies
\begin{equation}
\left(\int_{B}\left(\int_{\mathbb R^n}|f(t)e^{-2\pi (y\cdot t+\psi(y))}|^q dt\right)^{\frac {sp}{q}}dy\right)^{\frac{1}{sp}}\leq\|F\|_{A^{p,s}(B,\psi)} \label{2g01q}.
\end{equation}
Moreover, $f$ is supported in $U_s(B,\psi)$ for $p=1$ and supported in $U_{sp}(B,\psi)$ for $1<p\leq2$, $0<s(p-1)\leq 1$.
\end{theorem}

As given in the next theorem, integral representations in the form of Laplace transform are also available for $0<p<1$ and $0<s\leq\infty$.
\begin{theorem}
Assume that $F(z)\in A^{p,s}(B,\psi)$, where $0<p<1$ and $0<s\leq\infty$. Then there exists a continuous function $f(t)$ such that $f(t)e^{-2\pi y\cdot t}\in L^1(\mathbb R^n)$  and (\ref{mainconclusion}) hold for $y\in B$.
\end{theorem}

Considering the property of $f(t)$ for the case of $0<p<1$, we let $B$ be a regular open convex cone $\Gamma$ and let $\psi \in C(\Gamma)$ satisfy
\begin{equation}
 R_{\psi}=\varlimsup\limits_{y\in \Gamma,y\to\infty} \frac{\psi (y)}{|y|}<\infty.
 \label{psicondition}
\end{equation}
Then we obtain the following corollary.

\begin{corollary}
Assume that $\Gamma$ is a regular open convex cone and $F(z)\in A^{p,s}(\Gamma,\psi)$ for $0<p<1$, $0<s\leq\infty$, where $\psi\in C(\Gamma)$ satisfies (\ref{psicondition}). Then there exists $f(t)$ supported in $\Gamma^*+\overline{D(0,R_{\psi})}$ such that (\ref{mainconclusion}) holds and
$
|f(t)|\left(\int_{\Gamma}e^{-2s\pi (y\cdot t+R_{\psi}|y|)}dy\right)^{\frac1s}
$
is slowly increasing.
\end{corollary}

Similarly, we establish an analogy for $p>2$ and $0<s\leq\infty$.
\begin{theorem}
Assume that $p>2$, $0<s\leq\infty$, $\Gamma$ is a regular open convex cone in $\mathbb R^n$ and $\psi\in C(\Gamma)$ satisfies (\ref{psicondition}). If $F(z)\in A^{p,s}(\Gamma,\psi)$ satisfying
\begin{equation}
\varliminf_{y\in \Gamma,y\to0}\int_{\mathbb R^n}|F(x+iy)|^2dx<\infty,\label{hardylittlewood}
\end{equation}
then there exists $f(t)\in L^2(\mathbb{R}^n)$ supported in $U_{sp}(\Gamma,\psi)$ such that (\ref{mainconclusion})
  holds for all $z\in T_{\Gamma}$.
\end{theorem}

The definition of $A^{p,s}(B,\psi)$ shows that $A^{p,s}(B,\psi)$ is a weighted Hardy space when $s=\infty$ and a weighted Bergman space when $s=1$. Taking $\psi(y)=0$, it becomes, for $s=\infty$ and $s=1$, respectively, the classical Hardy space $H^p$ and the classical Bergman space $A^p$. Therefore, our results herein can be regarded as generalizations of certain previously obtained results.

For example, taking $s=\infty$ and $B$ a regular open convex cone $\Gamma$, $A^{p,\infty}(B,\psi)=H^p(\Gamma,\psi)$ is the weighted Hardy spaces investigated in our previous paper \cite{DHQ}. Then Theorem 1, 2 and 3 in \cite{DHQ} can be derived from our main work, including Theorem 1, 2, 3 and Corollary 1 herein.
For $s=\infty$ and $\psi(y)=0$, letting $B$ be some specific domains, some previous studies for the Hardy spaces, see \cite{SW,DQ,LI,Li,Q1}, can be also derived from Theorem 1, 2, 3 and Corollary 1.

On the other hand, letting $s=1$, by using Theorem 1, 2, 3 and Corollary 1, we can obtain the representation theorems for the standard Bergman spaces. Note that for $s=1$, $B=\Gamma$ and $\psi(y)=0$, we have $A^{p,s}(B,\psi)=A^p(T_{\Gamma})$. We therefore conclude from Theorem 1 that the counterpart results of Theorem 1, 2 and 3 in \cite{Gen5} hold for the classical Bergman spaces $A^p(T_{\Gamma})$($1\leq p\leq2$). If we set $\psi(y)=0$ and $s=q-1$, where $\frac1p+\frac1q=1$, then $A^{p,s}(B,\psi)=B^p(T_B)$. The integral representation theorems for those function spaces $B^p(T_B)$($1\leq p\leq 2$) can be derived from Theorem 1 herein, see \cite{Gen2}. Especially, letting $s=1$, $p=2$, $\psi (y)=-\frac{\alpha}{4\pi} \log |y|$ and $B$ a regular open convex cone $\Gamma$, Theorem 1 implies a higher dimensional generalization of Theorem 1 of \cite{PD} in tube domains, which is established as Corollary 2 in the sequel.

\section{Proofs}

This section is devoted to proving the results stated in \S2.

\begin{proof}[Proof of Theorem 1]
We first prove the case of $p=1$. If $F(z)\in A^{1,s}(B,\psi)$, then $F_y(x)\in L^1(\mathbb R^n)$ as a function of $x$, and $\check F_y(x)$ as well, are both well defined. Next we  prove that $\check F_y(t)e^{-2\pi y\cdot t}$ is independent of $y\in B$. Without loss of generality, assume that $a=(a',a_n)$, $b=(a',b_n)\in B$,  and  $a+\tau(b-a)\in B$ for $0\leq\tau\leq1$, where $a'=(a_1,\ldots,a_{n-1})$. The fact $F_y(x)\in L^1(\mathbb R^n)$ implies that
\begin{equation*}
\int_0^{\infty}\int_0^1\int_{\mathbb R^{n-1}}\left(|F((x',x_n)+i(a+\tau(b-a)))|+|F((x',-x_n)+i(a+\tau(b-a)))|\right)dx'd\tau dx_n<\infty,
\end{equation*}
which implies
\begin{equation*}
\varliminf_{R\to\infty}\int_0^1\int_{\mathbb R^{n-1}}\left(|F((x',R)+i(a+\tau(b-a)))|+|F((x',-R)+i(a+\tau(b-a)))|\right)dx'd\tau=0.
\end{equation*}
Hence, we have
\begin{eqnarray*}
&&|\check F_b(t)e^{-2\pi b\cdot t}-\check F_a(t)e^{-2\pi a\cdot t}|\\
&=&\left|\int_{\mathbb R^n}\left(F(x+ib)e^{2\pi i(x+ib)\cdot t}-F(x+ia)e^{2\pi i(x+ia)\cdot t}\right)dx\right|\\
&=&\left|\int_{\mathbb R^n}\int_0^1\frac{\partial}{\partial \tau}\left(F(x+i(a+\tau(b-a)))e^{2\pi i(x+i(a+\tau(b-a)))\cdot t}\right)d\tau dx\right|\\
&=&\left|\int_{\mathbb R^n}\int_0^1\frac{\partial}{\partial y_n}\left(F(x+i((y',y_n))e^{2\pi i(x+i(y',y_n))\cdot t}\big|_{y_n=a_n+\tau(b_n-a_n)}(b_n-a_n)\right)d\tau dx\right|\\
&=&|b_n-a_n|\left|\int_{\mathbb R^n}\int_0^1 i\frac{\partial}{\partial x_n}\left(F(x+i(a+\tau(b-a)))e^{2\pi i(x+i(a+\tau(b-a)))\cdot t}\right)d\tau dx\right|\\
&\leq&|b_n-a_n|\varliminf_{R\to\infty}\int_0^1\int_{\mathbb R^{n-1}}\left(\left|F((x',R),(a+\tau(b-a)))\right|+\left|F((x',-R),(a+\tau(b-a)))\right|\right)\\
&&e^{-2\pi|t|(|a|+|b-a|)}dx'd\tau\\
&=&0.
\end{eqnarray*}
Remark that $B$ is connected and open, by an iteration argument, we can show  that $\check F_y(t)e^{-2\pi y\cdot t}$ is independent of $y\in B$ and we write it as $g(t)$. Then
 $g(t)=\check F_{y }(t)e^{-2\pi y \cdot t}$ holds for $y\in B$.
Next, we show that $g(t)e^{2\pi y\cdot t}\in L^1(\mathbb R^n)$.
Let us decompose $\mathbb R^n$ into a finite union of non-overlapping polygonal cones, $\Gamma_1,\Gamma_2,\dots,\Gamma_N$ with their very vertexes at the origin, i.e., $\mathbb R^n=\bigcup_{k=1}^{N}\Gamma_k$. Then $
\chi_{\Gamma_k}(t)g(t)e^{2\pi y\cdot t}=\chi_{\Gamma_k}(t)\check F_{y_k}(t)e^{-2\pi(y_k-y)\cdot t}$.
For any $y_0\in B$, there exists $\delta$ such that $\overline{D(y_0,\delta)}\subset B$. Then for any $y\in D(y_0,\frac{\delta}{4})$ and $y_k\in (y_0+\Gamma_k)$ satisfying $\frac{3\delta}{4}\leq|y_k-y_0|<\delta$, we get $(y_k-y)\cdot t=(y_k-y_0)\cdot t+(y_0-y)\cdot t$.
Since $y_k-y_0,t\in\Gamma_k$, the angle between the segments $O(y_k-y_0)$ and $Ot$ is less than, say $\frac{\pi}{4}$. Then $(y_k-y)\cdot t\geq\frac{|y_k-y_0|}{\sqrt 2}|t|-|y_0-y||t|\geq(\frac{3}{4\sqrt 2}-\frac14)\delta|t|\geq\frac14\delta|t|$. Thus, it follows from H\"{o}lder's inequality that
\begin{eqnarray*}
\int_{\Gamma_k}|g(t)e^{2\pi y\cdot t}|dt\leq\int_{\Gamma_k}|\check F_{y_k}(t)e^{-\pi\frac{\delta}{4}|t|}|dt\leq\|F_{y_k}(x)\|_{L^1(\mathbb R^n)}\int_{\Gamma_k}e^{-\pi\frac{\delta}{4}|t|}dt<\infty,
\end{eqnarray*}
which shows that $g(t)e^{2\pi y\cdot t}\in L^1(\Gamma_k)$. Hence $g(t)e^{2\pi y\cdot t}\in L^1(\mathbb R^n)$. Together with the relation $g(t)=\check F_y(t)e^{-2\pi y\cdot t}$ for $y\in B$, there holds $F(z)=\int_{\mathbb R^n}g(t)e^{-2\pi iz\cdot t}$ for all $y\in B$. By letting $f(t)=g(-t)$, we then obtain the desired formula (\ref{mainconclusion}) for $p=1$ and $z\in T_B$.

Thus, $f(t)e^{-2\pi y\cdot t}\in L^1(\Gamma_k)$ implies that
\begin{eqnarray}
\sup_{t\in\mathbb R^n}|f(t)|e^{-2\pi y\cdot t}&\leq&\int_{\mathbb R^n}|F(x+iy)|dx \nonumber\\
|f(t)|e^{-2\pi y\cdot t}e^{-2\pi\psi(y)}&\leq&\int_{\mathbb R^n}|F(x+iy)|e^{-2\pi\psi(y)}dx \nonumber\\
|f(t)|^s\int_Be^{-2s\pi (y\cdot t+\psi(y))}dy&\leq&\int_B\left(\int_{\mathbb R^n}|F(x+iy)|e^{-2\pi\psi(y)}dx\right)^sdy \nonumber\\
&=&\|F\|^s_{A^{1,s}(B,\psi)}\label{p1bounded},
\end{eqnarray}
which implies (\ref{g0lq}). Next we prove supp$f \subset U_{s}(B,\psi)$. Suppose that $t_0\notin U_{s}(B,\psi)$, then (\ref{support}) implies $\int_{B}e^{-2s\pi (y\cdot t_0+\psi(y))}dy=+\infty$ for $y\in B$. It then follows from (\ref{p1bounded}) that $f(t)=0$, which means the support of $f$, i.e., supp$f\subset U_{s}(B,\psi)$.

Next we prove the case $1<p\leq2$.  Let $B_{0}\subseteq B$ be a bounded connected open set, so there exists a positive constant $R_0$ such that $B_0\subseteq  D(0,R_0)$. Assume that $l_{\varepsilon}(z)=(1+\varepsilon(z_1^2+\cdots+z_n^2))^N$, where $N$ is an integer satisfying  $N>n$. Then for $\varepsilon\leq\frac{1}{2R_0^2}, z=x+iy $ with $|y|\leq R_0$,
 \begin{eqnarray*}
|l_{\varepsilon}(z)|&=&|((1+\varepsilon(z_1^2+\cdots+z_n^2))^2)^{\frac N2}| \\
&=&\left(\left(1+\varepsilon(|x|^2-|y|^2)\right)^2+4\varepsilon^2
\left(x\cdot y\right)^2\right)^{\frac N2} \\
&\geq & \left(1+\varepsilon(|x|^2-|y|^2)\right)^N\geq\left(\frac12+\varepsilon|x|^2\right)^N
\end{eqnarray*}
for $|y|\leq R_0$, i.e., $\left|l^{-1}_{\varepsilon}(z)\right|\leq\frac{1}{\left(\frac12+\varepsilon|x|^2\right)^N}$.
For $F_y(x)=F(x+iy)$, set $F_{\varepsilon,y}(x)=F_y(x)l_{\varepsilon}^{-1}(z)$, then based on H\"{o}lder's inequality,
\begin{equation*}
\int_{\mathbb R^n}\left|F_{\varepsilon,y}(x)\right|dx\leq\left(\int_{\mathbb R^n}\left|F_{ y}(x)\right|^pdx\right)^{\frac1p}\left(\int_{\mathbb R^n}\left|l_{\varepsilon}^{-1}(x+iy)\right|^qdx\right)^{\frac1q}<\infty,
\end{equation*}
where  $\frac1p+\frac1q=1$,
which implies that $F_{\varepsilon,y}(x)\in L^1(\mathbb R^n)$. Then as in the proof for $p=1$, $g_{\varepsilon,y}(t)=\check F_{\varepsilon,y}(t)e^{-2\pi y\cdot t}$ can be also proved to be independent of $y\in B_0$ when $1<p\leq2$. Put $g_{\varepsilon,y}(t)=g_{\varepsilon}(t)$, then $g_{\varepsilon}(t)e^{2\pi y\cdot t}=\check F_{\varepsilon,y}(t)\in L^1(\mathbb R^n)$.

On the other hand, it is obvious that $F_{\varepsilon,y}(x)\to F_{y}(x)$ pointwise as $\varepsilon\to0$. Now we prove that $\check F_{y}(t)e^{-2\pi y\cdot t}$ is also independent of $y\in B_0$. Indeed, for $a,b\in B_0$ and any compact subset $K\subset \mathbb R^n$, let $R_1=\max \{|z|: z\in K\}$,
\begin{eqnarray*}
&&\left(\int_K\left|\check F_{a}(t)e^{-2\pi a\cdot t}-\check F_{b}(t)e^{-2\pi b\cdot t}\right|^qdt\right)^{\frac1q}\\
&\leq&\left(\int_K\left|\check F_{a}(t)e^{-2\pi a\cdot t}-g_{\varepsilon }(t) \right|^qdt\right)^{\frac1q}+\left(\int_K\left|g_{\varepsilon }(t) -\check F_{b}(t)e^{-2\pi b\cdot t}\right|^qdt\right)^{\frac1q}\\
&=&\left(\int_K\left|\check F_{a}(t)e^{-2\pi a\cdot t}-\check F_{\varepsilon,a}(t)e^{-2\pi a\cdot t}\right|^qdt\right)^{\frac1q}+\left(\int_K\left|\check F_{\varepsilon,b}(t)e^{-2\pi b\cdot t}-\check F_{b}(t)e^{-2\pi b\cdot t}\right|^qdt\right)^{\frac1q}\\
&\leq&e^{2\pi R_0 R_1}\left(\left(\int_K\left|\check F_{a}(t)-\check F_{\varepsilon,a}(t)\right|^qdt\right)^{\frac1q}+\left(\int_K\left|\check F_{\varepsilon,b}(t)-\check F_{b}(t)\right|^qdt\right)^{\frac1q}\right)\\
&\leq&e^{2\pi R_0 R_1}\left(\left(\int_{\mathbb R^n}\left|F_{a}(t)- F_{\varepsilon,a}(t)\right|^pdt\right)^{\frac1p}+\left(\int_{\mathbb R^n}\left| F_{\varepsilon,b}(t)- F_{b}(t)\right|^pdt\right)^{\frac1p} \right)\\
&\to&0,
\end{eqnarray*}
as $\varepsilon\to0$. Hence we obtain that $\check F_{a}(t)e^{-2\pi a\cdot t}=\check F_{b}(t)e^{-2\pi b\cdot t}$ almost everywhere on $\mathbb R^n$ and write it as $g(t)$. Then we have $g(t)=\check F_{y}(t)e^{-2\pi y\cdot t}$.

Next, we show that $g(t)e^{2\pi y\cdot t}\in L^1(\mathbb R^n)$. As in the proof for $p=1$, let $\mathbb R^n=\bigcup_{k=1}^{N}\Gamma_k$ and $\overline{D(y_0,\delta)}\subset B_0$. Then for any $y\in D(y_0,\frac{\delta}{4})$ and $y_k\in (y_0+\Gamma_k)$ satisfying $\frac{3\delta}{4}\leq|y_k-y_0|<\delta$, we have
$$
 (y_k-y)\cdot t\geq\frac{|y_k-y_0|}{\sqrt 2}|t|-|y_0-y||t|\geq(\frac{3}{4\sqrt 2}-\frac14)\delta|t|\geq\frac14\delta|t|
$$
for $y_k-y_0,t\in\Gamma_k$. Thus, from H\"{o}lder's inequality
\begin{eqnarray*}
\int_{\Gamma_k}|g(t)e^{2\pi y\cdot t}|dt&\leq&\int_{\Gamma_k}|\check F_{y_k}(t)e^{-\pi\frac{\delta_0}{4}|t|}|dt\leq\left(\int_{\Gamma_k}|\check F_{y_k}(t)|^pdt\right)^{\frac1p}\left(\int_{\Gamma_k}|e^{-q\pi\frac{\delta_0}{4}|t|}|dt\right)^{\frac1q}<\infty,
\end{eqnarray*}
which shows that $g(t)e^{2\pi y\cdot t}\in L^1(\Gamma_k)$ and the function $G(z)$  defined by
$$
G(z)=\int_{\mathbb R^n}g(t)e^{-2\pi i(x+iy)\cdot t}dt
$$
is holomorphic in the tube domain  $T_{ D(y_0,\delta)}$.

Now we can prove that, for $y\in B_{0}$,
$$
\lim_{ \varepsilon\to0}\int_{\mathbb R^n}g_{\varepsilon}(t)e^{-2\pi i(x+iy)\cdot t}dt= \int_{\mathbb R^n}g(t)e^{-2\pi i(x+iy)\cdot t}dt.
$$
In fact, if $y\in B_{0}$,
\begin{eqnarray*}
&&\left|\int_{\mathbb R^n}(g_{\varepsilon}(t)-g(t))e^{-2\pi i(x+iy)\cdot t}dt\right|\\
&\leq& \int_{\mathbb R^n}\left|\left(\check F_{\varepsilon,y}(t)e^{-2\pi y\cdot t}-\check F_{y}(t)e^{-2\pi y\cdot t}\right)e^{2\pi iz\cdot t}\right|dt\\
&=&\sum_{k=1}^n\int_{\Gamma_k}\left|\left(\check F_{\varepsilon,y_k}(x)-\check F_{y_k}(x)\right)e^{-2\pi i(y_k-y)\cdot  t}\right|dt\\
&\leq& \sum_{k=1}^n\left(\int_{\Gamma_k}|\check F_{\varepsilon,y_k}(x)-\check F_{y_k}(x)|^qdt\right)^{\frac1q} \left(\int_{\Gamma_k}e^{-p\pi\frac{\delta_0}{4}|t|}dt\right)^{\frac1p}\\
&\leq&C_{\delta_0}\sum_{k=1}^n\left(\int_{\Gamma_k}| F_{\varepsilon,y_k}(x)-F_{y_k}(x)|^pdt\right)^{\frac1p}\\
&\to&0
\end{eqnarray*}
when $\varepsilon\to0$, where $C_{\delta_0}^p=\int_{\mathbb{R}^n}e^{-p\pi\frac{\delta_0}{4}|t|}dt$. It follows that $\lim\limits_{\varepsilon\to0}F_{\varepsilon}(z)=G(z)$. Combining with $\lim\limits_{\varepsilon\to0}F_{\varepsilon}(z)=F(z)$, we can state $G(z)=F(z)$ for $y\in B_{0}$. Then there exists a measurable function $g(t)$ such that $F(z)=\int_{\mathbb R^n}g(t)e^{-2\pi iz\cdot t}dt$ holds for $ y\in B_{0}$. Since $B$ is connected, we can choose a sequence of bounded connected open set   $\{B_k\}$  such that $B_{0}\subset B_{1}\subset\cdots$ and $B=\bigcup_{k=0}^{\infty}B_{k}$. Together with the fact that $g(t)=\check F_y(t)e^{-2\pi y\cdot t}$ is independent of $y\in B_{k}$, then $\check F_{y_l}(t)e^{-2\pi y_l\cdot t}=\check F_{y_j}(t)e^{-2\pi y_j\cdot t}=\check F_{y }(t)e^{-2\pi y \cdot t}$ for $l\not=j$, $y_l\in B_{l}, y_j\in B_j$ and $y\in B_0$. Hence $g(t)e^{2\pi y\cdot t}=\check F_y(t)$ holds for $y\in B_{k}$, $k=0,1,2,\ldots.$ In other words, $f(z)=\int_{\mathbb R^n}g(t)e^{-2\pi iz\cdot t}dt$ holds for all $y\in B$. By letting $f(t)=g(-t)$, we obtain the desired representation $F(z)=\int_{\mathbb R^n}f(t)e^{2\pi iz\cdot t}dt$ for $y\in B$ when $1<p\leq 2$.

For $\frac1p+\frac1q=1$, based on the Hausdorff-Young Inequality,
\begin{equation}
\left(\int_{\mathbb R^n}|f(t)e^{-2\pi y\cdot t}|^qdt\right)^{\frac1q}\leq\left(\int_{\mathbb R^n}|F(x+iy)|^pdx\right)^{\frac1p},\label{1p2hyi}
\end{equation}
then
\begin{equation*}
\left(\left(\int_{\mathbb R^n}|f(t)e^{-2\pi y\cdot t}|^qdt\right)^{\frac pq}e^{-2p\pi\psi(y)}dy\right)^s\leq\left(\int_{\mathbb R^n}|F(x+iy)e^{-2\pi\psi(y)}|^pdx\right)^s.
\end{equation*}
Performing integral about $y\in B$ on both sides, we get
\begin{equation*}
\int_{B}\left(\left(\int_{\mathbb R^n}|f(t)e^{-2\pi y\cdot t}|^qdt\right)^{\frac pq}e^{-2p\pi\psi(y)}\right)^sdy\leq \int_{B}\left(\int_{\mathbb R^n}|F(x+iy)e^{-2\pi\psi(y)}|^pdx\right)^sdy
\end{equation*}
and
\begin{equation}
\int_{B}\left(\left(\int_{\mathbb R^n}|f(t)e^{-2\pi y\cdot t}|^qdt\right)^{\frac pq}e^{-2p\pi\psi(y)}\right)^sdy\leq \|F\|^{sp}_{A^{p,s}(B,\psi)}.\label{1p2hyi-9}
\end{equation}
As a result, formulas (\ref{mainconclusion}) and (\ref{2g01q}) hold for $1<p\leq2$. Now we prove that supp$f\subset U_{sp}(B,\psi)$ when $0<s(p-1)\leq 1$. For $0<s(p-1)\leq 1$, we have
$\frac{q}{sp}\geq 1$. Then Minkowski's inequality and (\ref {1p2hyi-9}) imply that
\begin{equation}
\int_{\mathbb R^n}|f(t)|^q\left(\int_Be^{-2\pi ps(y\cdot t+\psi(y))}dy\right)^{\frac {q}{ps}}dt\leq \|F\|^{q}_{A^{p,s}(B,\psi)}<\infty.\label{p12support}
\end{equation}
Consequently, It follows from (\ref{p12support}) and (\ref{support}) that $f(t)=0$ for almost every $ t \not\in U_{sp}(B,\psi)$. Therefore, supp$f\subset U_{sp}(B,\psi)$.
\end{proof}

In order to prove Theorem 2, we first introduce a lemma.
\begin{lemma}\label{subharmonicproperty}
Suppose that $F(z)\in A^{p,s}(B,\psi)$, where $0<p<\infty$ and $0<s\leq\infty$, then for $y_0\in B$ and   positive constant $\delta$ such that $D_n(y_0,\delta)\subset B$ , there exist constants $N>1$ and $C_{n,N,p,s}$ depending on $n, N, p, s$ such that
\begin{equation}
|F(z)|\leq C_{n,N,p,s}\delta^{-\frac np(1+\frac1s)}e^{2 \pi\psi_{\delta}(y_0)}, \label{Fvarepsilonrange}
\end{equation}
where $\psi_{\delta}(y_0)=\max\{\psi(\eta):|\eta-y_0|\leq\delta\}$.
\end{lemma}
\begin{proof}
For $y_0\in B$, there exists $\delta>0$ such that $B_{\delta}=D(y_0,\delta)\subset B$. Then for $F(z)=F(x+iy)\in A^{p,s}(B,\psi)$, based on the subharmonic properties of $|F(z)|^t$, we have
\begin{eqnarray*}
|F(z)|^t&\leq&\frac{1}{\Omega_{2n}\delta^{2n}}\int_{D_{2n}(z,\delta)}|F(\xi+i\eta)|^td\xi d\eta\nonumber\leq\frac{1}{\Omega_{2n}\delta^{2n}}\int_{D_n(y_0,\delta)}\left(\int_{D_n(x,\delta)}|F(\xi+i\eta)|^t d\xi\right)d\eta
\end{eqnarray*}
for $y\in B_{\delta}$, where $\Omega_k$ is the volume of $k$-dimensional unit ball $D_k(0,1)$ centered at the origin with radius 1, $k=n,2n$.
Let $p_1=N=\frac pt>\max\{1,\frac1s\}$ and $\frac{1}{p_1}+\frac{1}{q_1}=1$. H\"{o}lder's Inequality implies that
\begin{eqnarray*}
|F(z)|^t&\leq&\frac{1}{\Omega_{2n}\delta^{2n}}\int_{D_n(y_0,\delta)}\left(\int_{D_n(x,\delta)}|F(\xi+i\eta)|^p d\xi\right)^{\frac{1}{p_1}}d\eta\left(\int_{D_n(x,\delta)}1^{q_1}d\xi\right)^{\frac{1}{q_1}}\\
&=&\frac{(\delta^n\Omega_n)^{\frac{1}{q_1}}}{\delta^{2n}\Omega_{2n}}\int_{D_n(y_0,\delta)}\left(\int_{D_n(x,\delta)}|F(\xi+i\eta)|^p d\xi\right)^{\frac{1}{p_1}}d\eta.
\end{eqnarray*}
For $0<s<\infty$, let $p_2=sN$. Then $p_2>1$. Again, by H\"{o}lder's Inequality, for $\frac{1}{p_2}+\frac{1}{q_2}=1$,
\begin{eqnarray*}
|F(z)|^t&\leq&\frac{(\delta^n\Omega_n)^{\frac{1}{q_1}}}{\delta^{2n}\Omega_{2n}}\left(\int_{D_n(y_0,\delta)}\left(\int_{D_n(x,\delta)}|F(\xi+i\eta)|^p d\xi\right)^{s}d\eta\right)^{\frac{1}{p_2}}\left(\int_{D_n(y_0,\delta)}1^{q_2}d\eta\right)^{\frac{1}{q_2}}\\
&\leq&\frac{(\delta^n\Omega_n)^{\frac{1}{q_1}+\frac{1}{q_2}}}{\delta^{2n}\Omega_{2n}}\left(\int_{D_n(y_0,\delta)}\left(\int_{D_n(x,\delta)}|F(\xi+i\eta)e^{-2\pi\psi(\eta)}|^p d\xi\right)^{s}e^{2sp\pi\psi(\eta)}d\eta\right)^{\frac{1}{p_2}}\\
&\leq&\frac{(\delta^n\Omega_n)^{2-\frac1N(1+\frac1s)}e^{2\frac{sp}{p_2}\pi\psi_{\delta}(y_0)}}{\delta^{2n}\Omega_{2n}}\left(\int_{D_n(y_0,\delta)}\left(\int_{D_n(x,\delta)}|F(\xi+i\eta)e^{-2\pi\psi(\eta)}|^p d\xi\right)^{s}d\eta\right)^{\frac{1}{p_2}}\\
&\leq&\frac{(\delta^n\Omega_n)^{2-\frac1N(1+\frac1s)}e^{2\frac{sp}{p_2}\pi\psi_{\delta}(y_0)}}{\delta^{2n}\Omega_{2n}}\left(\int_{B}\left(\int_{\mathbb R^n}|F(\xi+i\eta)e^{-2\pi\psi(\eta)}|^p d\xi\right)^{s}d\eta\right)^{\frac{1}{p_2}},
\end{eqnarray*}
where $\psi_{\delta}(y_0)=\max\{\psi(\eta):|\eta-y_0|\leq\delta\}$. Hence,
\begin{eqnarray*}
|F(z)|&\leq&\left(\frac{\delta^{-\frac nN(1+\frac1s)}\Omega_n^{2-\frac1N(1+\frac1s)} e^{2\frac{sp}{p_2}\pi\psi_{\delta}(y_0)}}{\Omega_{2n}}\right)^{\frac1t} \left(\int_{B}\left(\int_{\mathbb R^n}|F(\xi+i\eta)e^{-2\pi\psi(\eta)}|^p d\xi\right)^{s}d\eta\right)^{\frac{1}{t p_2}}\\
&\leq&\frac{\Omega_n^{\frac{2N}{p}-\frac1p(1+\frac1s)}}{\Omega_{2n}^{\frac Np}\delta^{\frac np(1+\frac1s)}}e^{2\frac{sp}{tp_2}\pi\psi_{\delta}(y_0)}\left(\int_{B}\left(\int_{\mathbb R^n}|F(\xi+i\eta)e^{-2\pi\psi(\eta)}|^p d\xi\right)^{s}d\eta\right)^{\frac{1}{sp}\frac{sp}{t p_2}}.
\end{eqnarray*}
Since $\frac{sp}{tp_2}=1$, by letting $C_{n,N,p,s}=\frac{\Omega_n^{\frac{2N}{p}-\frac1p(1+\frac1s)}}{\Omega_{2n}^{\frac Np}}\|F(z)\|_{A^{p,s}(B,\psi)}$, we obtain the desired inequality
\begin{equation*}
|F(z)|\leq C_{n,N,p,s}\delta^{-\frac np(1+\frac1s)}e^{2 \pi\psi_{\delta}(y_0)}.
\end{equation*}

While $s=\infty$, for $p_2=sN=\infty$, we have
\begin{eqnarray*}
|F(z)|^t\leq\frac{(\delta^n\Omega_n)^{2-\frac1N}}{\delta^{2n}\Omega_{2n}}\sup\limits_{\eta\in D_n(y,\delta)}\left|\int_{D_n(x,\delta)}|F(\xi+i\eta)|^pd\xi\right|^{\frac tp}.
\end{eqnarray*}
Then
\begin{eqnarray*}
|F(z)|&\leq&\frac{(\delta^n\Omega_n)^{(2-\frac1N)\frac Np}}{(\delta^{2n}\Omega_{2n})^{\frac Np}}e^{2\pi\psi_{\delta}(y_0)}\sup\limits_{\eta\in D_n(y,\delta)}\left|\left(\int_{D_n(x,\delta)}|F(\xi+i\eta)|^pd\xi\right)^{\frac 1p}e^{-2\pi\psi(y)}\right|\\
&=&\frac{\Omega_n^{\frac{2N}{p}-\frac1p}}{\Omega_{2n}^{\frac Np}}\delta^{-\frac np}e^{2\pi\psi_{\delta}(y_0)}\|F(z)\|_{A^{p,\infty}(B,\psi)}.
\end{eqnarray*}
Obviously, the inequality (\ref{Fvarepsilonrange}) is also applicable in the case $s=\infty$.
\end{proof}

Now we are ready to prove Theorem 2.
\begin{proof}[Proof of Theorem 2]
For $y_0\in B$, there exists $\delta>0$ such that $B_{\delta}=D(y_0,\delta)\subset B$. Then for $F(z)\in A^{p,s}(B,\psi)$ and any $y\in B_{\delta}$, it follows from Lemma \ref{subharmonicproperty} that
\begin{equation*}
|F(z)|\leq C_{n,N,p,s}\delta^{-\frac np(1+\frac1s)}e^{2\pi\psi_{\delta}(y_0)}.
\end{equation*}
Thus,
\begin{equation*}
\int_{\mathbb R^n}|F(z)|^2dx=\int_{\mathbb R^n}|F(z)|^{p+2-p}dx\leq C^{2-p}_{n,N,p,s}\delta^{-\frac {n(2-p)}{p}(1+\frac1s)}e^{2(2-p)\pi\psi_{\delta}(y_0)}\int_{\mathbb R^n}|F(z)|^pdx.
\end{equation*}
Therefore,
\begin{eqnarray*}
&&\int_{\mathbb R^n}|F(z)|^2e^{-4\pi\psi_{\delta}(y_0)}dx\\
&\leq& C^{2-p}_{n,N,p,s}\delta^{-\frac {n(2-p)}{p}(1+\frac1s)}e^{2(2-p)\pi\psi_{\delta}(y_0)}\int_{\mathbb R^n}|F(z)e^{-2\pi\psi(y)}|^pdxe^{2p\pi\psi(y)}e^{-4\pi\psi_{\delta}(y_0)}\\
&\leq& C^{2-p}_{n,N,p,s}\delta^{-\frac {n(2-p)}{p}(1+\frac1s)}e^{2(2-p)\pi\psi_{\delta}(y_0)}\int_{\mathbb R^n}|F(z)e^{-2\pi\psi(y)}|^pdxe^{2(p-2)\pi\psi_{\delta}(y_0)}\\
&=&C^{2-p}_{n,N,p,s}\delta^{-\frac {n(2-p)}{p}(1+\frac1s)}\int_{\mathbb R^n}|F(z)e^{-2\pi\psi(y)}|^pdx.
\end{eqnarray*}
Taking integral with respect to $y$ to both sides of the inequality, we have
\begin{equation*}
\int_{B_{\delta}}\left(\int_{\mathbb R^n}|F(z)|^2 e^{-4\pi\psi_{\delta}(y_0)}dx\right)^sdy\leq C^{(2-p)s}_{n,N,p,s}\delta^{-\frac {n(2-p)(1+s)}{p}}\int_{B_{\delta}}\left(\int_{\mathbb R^n}|F(z)e^{-2\pi\psi(y)}|^pdx\right)^sdy,
\end{equation*}
which concludes that $F\in A^{2,s}(B_{\delta},\psi_{\delta})$.
Similarly, we can prove that
\begin{equation}
\int_{B_{\delta}}\left(\int_{\mathbb R^n}|F| e^{-2\pi\psi_{\delta}(y_0)}dx\right)^sdy\leq C^{(1-p)s}_{n,N,p,s}\delta^{-\frac {n(1-p)(1+s)}{p}}\|F(z)\|^{sp}_{A^{1,s}(B_{\delta},\psi)}.\label{Fbergman1}
\end{equation}
Then $F(z)\in A^{1,s}(B_{\delta},\psi_{\delta})$.

Following the proof of the case $p=1$ in Theorem 1, there exists a continuous function $f (t)$ such that $F_y(x)=\int_{\mathbb R^n}f (t)e^{2\pi iz\cdot t}dt$ holds for $y\in B_{\delta}$ and  $f (t)=\hat F_y(t)e^{ 2\pi y\cdot t}$ is independent of $y\in B$.
 Together with the fact that $f(t)e^{-2\pi y\cdot t}\in L^1(\mathbb R^n)$ for all $y\in B$, we see  that (\ref{mainconclusion})  holds for all $y\in B$. This completes the proof
 of Theorem 2.
\end{proof}

Before the proof of Corollary 1, we introduce the following lemma.
\begin{lemma}\label{supportrelation}
Assume that $\Gamma$ is a regular open convex cone of $\mathbb R^n$. Let $\psi\in C(\Gamma)$ satisfy (\ref{psicondition}), then  $U_{\alpha}(\psi,\Gamma)\subset \Gamma^*+\overline{D(0,R_{\psi})}$, where  $U_{\alpha}(\psi,\Gamma)$ is defined by (\ref{support}) for $0<\alpha<\infty$ and by (\ref{support-2}) for $\alpha=\infty$.
\end{lemma}
\begin{proof}
For $t_0\notin\Gamma^*+\overline{D(0,R_{\psi})}$, there exist $\varepsilon>0$ and $\xi\in\Gamma^*$ such that
$
d(t_0,\Gamma^*)=|\xi-t_0|\geq R_{\psi}+3\varepsilon
$
and $\xi\cdot (t_0-\xi)=0$. Then for any $\tilde t\in\Gamma^*$,
\begin{equation*}
(\tilde t-t_0)\cdot\frac{(\xi-t_0)}{|\xi-t_0|}\geq|\xi-t_0|.
\end{equation*}
Hence $\tilde t\cdot(\xi-t_0)=(\tilde t-t_0+t_0-\xi+\xi)\cdot(\xi-t_0)\geq|\xi-t_0|^2-|\xi-t_0|^2=0$,
which means $\xi-t_0\in\overline\Gamma$. For any $\delta>0$, it follows from (\ref{psicondition}) that there exists $\rho_0$ such that $\psi(y)\leq (R_{\psi}+\delta)|y|$ for $|y|\geq\rho_0$.
Let $e_0=\frac{\xi-t_0}{|\xi-t_0|}\in \overline\Gamma\cap\partial D(0,1)$, then for any $\varepsilon_1>0$, we can find an $e_1\in\Gamma$ with $|e_1|=1$ such that $|e_1-e_0|<\varepsilon_1$, which means there exists a positive constant $\delta_1<\varepsilon_1$ such that $D(e_1,\delta_1)\subset\Gamma$. Thus, for any $e\in D(e_1,\delta_1)$ with $|e_1|=1$, we have $|e-e_0|\leq|e-e_1|+|e_1-e_0|<2\varepsilon_1$. Choose $\varepsilon_1$ satisfying $2\varepsilon_1|t_0|\leq\varepsilon$ and let $\Gamma_1=\{y=\rho e:\rho>0\mbox{ and }e\in D(e_1,\delta)\cap\partial D(0,1)\}\subset\Gamma$. Then for any $y\in\Gamma_1$,
$
-\rho e\cdot t_0=\rho(-e+e_0-e_0)\cdot t_0\geq\rho(-2\varepsilon_1|t_0|+|\xi-t_0|)\geq \rho(R_{\psi}+2\varepsilon)
$
and
\begin{eqnarray*}
\int_{\Gamma}e^{-2\pi \alpha(t_0\cdot y+\psi(y))}dy
&\geq&\int_{\Gamma\cap\{|y|\geq\rho_0\}}e^{-2\pi \alpha(t_0\cdot y+(R_{\psi}+\delta)|y|)}dy\\
&\geq&\int_{\rho_0}^{\infty}\rho^{n-1}d\rho \int_{\partial D(0,1)\cap D(e_1,\delta_1)}e^{2\pi \alpha\rho(2\varepsilon-\delta)}d\sigma(\zeta)=+\infty,
\end{eqnarray*}
which implies $t_0\notin U_{\alpha}(\psi,\Gamma)$. Therefore, $U_{\alpha}(\psi,\Gamma)\subset \Gamma^*+\overline{D(0,R_{\psi})}$.
\end{proof}
Now we prove Corollary 1.
\begin{proof}[Proof of Corollary 1]
For $y_0\in\Gamma$, there exists $\delta $ such that $D(y_0,\delta )\subset\Gamma$. It follows from Theorem 2 that there exists $f(t)$ such that (\ref{mainconclusion}) holds for $y\in D(y_0,\delta )$. Since $\Gamma$ is connected, (\ref{mainconclusion}) also holds for all $y\in \Gamma$. Applying the methods in the proof of Theorem 1 for $p=1$, we obtain that such an $f(t)$ is supported in $U_s(\Gamma,\psi_{\delta})$. Combing with Lemma \ref{supportrelation}, we have supp$f\subset U_s(\Gamma,\psi_{\delta})\subset \Gamma^*+\overline{D(0,R_{\psi_{\delta}})}$,
where
\begin{equation*}
 R_{\psi_\delta}=\varlimsup\limits_{y\in \Gamma,y\to\infty} \frac{\psi_\delta (y)}{|y|}.
\end{equation*}
Since $R_{\psi_{\delta}}=R_{\psi}$ for any $y\in\Gamma$, we see that $U_s(\Gamma,\psi_{\delta})$ is also a subset of $\Gamma^*+\overline{D(0,R_{\psi})}$. Hence, supp$f\subset \Gamma^*+\overline{D(0,R_{\psi})}$.

Now we show that $
|f(t)|\left(\int_{\Gamma}e^{-2s\pi (y\cdot t+R_{\psi}|y|)}dy\right)^{\frac1s}
$
is slowly increasing.
For $y_0$, $y\in\Gamma$, $y_0+y\in\Gamma$, $F_{y_0}(z)=F(x+i(y+y_0))\in A^{p,s}(\Gamma,\psi)$. As in Theorem 1, we have $f(t)=g(-t)=\check F_{y_0+y}(-t)e^{2\pi(y_0+y)\cdot t}$.
Due to the relation $R_{\psi}=\varlimsup\limits_{y\in B,y\to\infty}\frac{\psi(y)}{|y|}$, we have $\psi_{\delta}(y)\leq R_{\psi}(1+|y_0|+|y|)$, where $R_{\psi}$ is a positive constant independent of $y_0$, $y\in \Gamma$. Then
\begin{eqnarray*}
|f(t)|&=&|\check F_{y_0+y}(-t)e^{2\pi(y_0+y)\cdot t}|= \left|\int_{\mathbb R^n}F_{y_0+y}(x)e^{-2\pi ix\cdot t}e^{-2\pi \psi_{\delta}(y)}dx\right| e^{2\pi(\psi_{\delta}(y)+(y_0+y)\cdot t)}\\
&\leq&\int_{\mathbb R^n}|F_{y_0}(z)|e^{-2\pi \psi_{\delta}(y)}dx e^{2\pi (R_{\psi}(1+|y_0|+|y|)+(y_0+y)\cdot t)}.
\end{eqnarray*}
Combining with (\ref{Fbergman1}), it follows that
\begin{eqnarray*}
\left(\int_{\Gamma}|f(t)|^se^{-2s\pi (y\cdot t+R_{\psi}|y|)}dy\right)^{\frac1s}&\leq&\left(\int_{\Gamma}\left(\int_{\mathbb R^n}|F_{y_0}(z)|e^{-2\pi \psi_{\delta}(y)}dx\right)^sdy\right)^{\frac1s} e^{2 \pi (R_{\psi}(1+|y_0|)+y_0\cdot t)}\\
&\leq& C_{n,N, p,s}^{1-p}\delta^{-\frac{n(1-p)(1+s)}{sp}}\|F_{y_0}\|^p_{A^{1,s}(B,\psi)}e^{2\pi (R_{\psi}(1+|y_0|)+y_0\cdot t)}\\
&=&C\exp\{J(y_0,t)\},
\end{eqnarray*}
where $C=C_{n,N, p,s}^{1-p}\|F_{y_0}\|^p_{A^{1,s}(\Gamma,\psi)}$ and $J(y_0)=-\frac{n(1-p)(1+s)}{sp}\log\delta+2\pi (R_{\psi}(1+|y_0|)+y_0\cdot t)$. Let $J(t)=\inf\{J(y_0,t):y_0\in \Gamma\}$, then
\begin{equation*}
|f(t)|\left(\int_{\Gamma}e^{-2s\pi (y\cdot t+R_{\psi}|y|)}dy\right)^{\frac1s}\leq C\exp\{J(t)\}.
\end{equation*}

Take $y_0=\rho v$ with $\rho>0$ and a fixed $v\in \Gamma$ with $|v|=1$, then $\delta=d(\rho v,\partial \Gamma)/2=\rho\varepsilon$, where $\varepsilon=d(v,\partial \Gamma)/2$. Therefore,
\begin{equation*}
J(t)=\inf\limits_{\rho>0}\left\{-\frac{n(1-p)(1+s)}{sp}\log(\varepsilon\rho)+2\pi R_{\psi}(1+\rho)+2\pi\rho|t|\right\},
\end{equation*}
in which the infimum can be attained when $\rho=\frac{n(1-p)(1+s)}{2sp\pi (R_{\psi}+|t|)}$. It follows that
\begin{equation*}
J(t)\leq 2\pi R_{\psi}+ n\left(\frac1p-1\right)\left(\frac1s+1\right)\left(1-\log\varepsilon-\log n\left(\frac1p-1\right)\left(\frac1s+1\right) +\log 2\pi (R_{\psi}+|t|) \right).
\end{equation*}
Thus, there exists a positive constant $M_{n,p,s,v}$ such that
\begin{equation*}
|f(t)|\left(\int_{\Gamma}e^{-2s\pi (y\cdot t+R_{\psi}|y|)}dy\right)^{\frac1s}\leq Ce^{J(t)}\leq M_{n,p,s,v}(1+|t|)^{n(\frac1p-1)(\frac1s+1)}.
\end{equation*}
The proof is complete.
\end{proof}

\begin{proof}[Proof of Theorem 3]
We first prove the case when $2<p<\infty$. Since $\Gamma$ is a regular open convex cone, int$\Gamma\not=\emptyset$, where int$\Gamma$ is denoted as the interior of $\Gamma$. Then for $y\in\Gamma$, we can find a basis $\{e_j\}\subset \mbox{int}\Gamma^*$ such that $y=\sum_{j=1}^n e_jy_j$ and $e_j\cdot y\geq0$. For $\varepsilon>0$, let $l_{\varepsilon}(z)=\left(\prod_{j=1}^n(1-i\varepsilon e_j\cdot z)\right)^{2N}$ with $N>\frac n2\left(1-\frac1p\right)$ and choose two positive constant $A$, $B$ such that $B|x|^2\leq\varepsilon^2\sum_{j=1}^n(e_j\cdot x)^2\leq A|x|^2$ for all $x\in\mathbb R^n$. Thus,
\begin{eqnarray*}
|l_{\varepsilon}(z)|&=&\left(\prod_{j=1}^n|1-i\varepsilon e_j\cdot z|^2\right)^N=\left(\prod_{j=1}^n\left((1+\varepsilon e_j\cdot y)^2+\varepsilon^2(e_1\cdot x)^2\right)\right)^N\\
&\geq&\left(\prod_{j=1}^n\left(1+\varepsilon^2(e_j\cdot x)^2\right)\right)^N\geq\left(1+\varepsilon^2\sum_{j=1}^n(e_j\cdot x)^2\right)^N\geq\left(1+\varepsilon^2 B|x|^2\right)^N,
\end{eqnarray*}
i.e., $|l_{\varepsilon}^{-1}(z)|\leq\left(1+\varepsilon^2 B|x|^2\right)^{-N}$. For $F(x+iy)\in A^{p,s}(\Gamma,\psi)$, $F_y(x)=F(x+iy)\in L^p(\mathbb R^n)$ as a function of $x$. Let $F_{\varepsilon}(z)=F_{\varepsilon,y}(x)=F_y(x)l^{-1}_{\varepsilon}(z)$, then $F_{\varepsilon,y}(x)\in L^1(\mathbb R^n)\cap L^2(\mathbb R^n)$. Indeed, H\"{o}lder's inequality implies that
\begin{equation}
\int_{\mathbb R^n}|F_{\varepsilon,y}(x)|dx\leq\left(\int_{\mathbb R^n}|F_y(x)|^pdx\right)^{\frac 1p}\left(\int_{\mathbb R^n}|l_{\varepsilon}^{-1}(x+iy)|^qdx\right)^{\frac1q}\leq C_{1,\varepsilon}\|F_y\|_{L^p(\mathbb R^n)} \label{FyC1}
\end{equation}
and
\begin{equation*}
\int_{\mathbb R^n}|F_{\varepsilon,y}(x)|^2dx\leq\left(\int_{\mathbb R^n}|F_y(x)|^{\frac p2}dx\right)^{\frac 2p}\left(\int_{\mathbb R^n}|l_{\varepsilon}^{-1}(x+iy)|^{\frac{p}{p-2}}dx\right)^{\frac{p-2}{p}}\leq C_{2,\varepsilon}\|F_y\|_{L^p(\mathbb R^n)},
\end{equation*}
where $C_{1,\varepsilon}=\left(\int_{\mathbb R^n}\frac{dx}{\left(1+\varepsilon^2B|x|^2\right)^{qN}}\right) ^{\frac1q}<\infty$, $C_{2,\varepsilon}=\left(\int_{\mathbb R^n}\frac{dx}{\left(1+\varepsilon^2B|x|^2\right)^{\frac{p}{p-2}N}}\right) ^{\frac{p-2}{p}}<\infty$.

As the proof of  $p=1$ in Theorem 1, we can show $g_{\varepsilon}(t)e^{2\pi y\cdot t}=\check F_{\varepsilon,y}(t)\in L^1(\mathbb R^n)$. Thus,
\begin{equation}
g_{\varepsilon}(t)e^{2\pi y\cdot t}=\int_{\mathbb R^n}F_{\varepsilon,y}(x)e^{2\pi ix\cdot t}dx,\label{gFlaplace}
\end{equation}
then $|g_{\varepsilon}(t)|e^{2\pi y\cdot t}\leq\int_{\mathbb R^n}|F_{\varepsilon, y}(x)|dx$. Together with (\ref{FyC1}), there hold
\begin{eqnarray*}
|g_{\varepsilon}(t)|e^{2\pi(y\cdot t-\psi(y))}&\leq& C_{1,\varepsilon}\left(\int_{\mathbb R^n}|F(x+iy)e^{-2\pi\psi(y)}|^pdx\right)^{\frac 1p},\\
\int_{\Gamma}|g_{\varepsilon}(t)|^{sp}e^{2sp\pi(y\cdot t-\psi(y))}dy &\leq&C_{1,\varepsilon}\int_{\Gamma}\left(\int_{\mathbb R^n}|F(x+iy)e^{-2\pi\psi(y)}|^pdx\right)^sdy,\\
|g_{\varepsilon}(t)|^{sp}\int_{\Gamma}e^{2sp\pi(y\cdot t-\psi(y))}dy &\leq&C_{1,\varepsilon}\|F\|^{sp}_{A^{p,s}(\Gamma,\psi)}.
\end{eqnarray*}
Now we prove that supp$g_{\varepsilon}(t)\subset -U_{ps}(\Gamma,\psi)$. Note that $g_{\varepsilon}(t)$ is continuous in $\mathbb R^n$. Then for $t_0\notin -U_{ps}(\Gamma,\psi)$, formula (\ref{support}) shows that $\int_{\Gamma}e^{2ps\pi(y\cdot t_0-\psi(y))}dy=\infty$ for $y\in\Gamma$. It follows from the above inequality that $g_{\varepsilon}(t_0)=0$ for $t_0\notin -U_{ps}(\Gamma,\psi)$. As a result, supp$g_{\varepsilon}(t)\subset -U_{ps}(\Gamma,\psi)$.

Since $g_{\varepsilon}(t)e^{2\pi y\cdot t}\in L^1(\mathbb R^n)$, we can rewrite (\ref{gFlaplace}) as
\begin{equation}
F_{\varepsilon,y}(x)=\int_{\mathbb R^n}g_{\varepsilon}(t)e^{-2\pi iz\cdot t}\chi_{-U_{ps}(\Gamma,\psi)}(t)dt.\label{chi}
\end{equation}
Plancherel's Theorem implies that $\int_{\mathbb R^n}|g_{\varepsilon}(t)e^{2\pi y\cdot t}|^2dt=\int_{\mathbb R^n}|F_{\varepsilon,y}(x)|^2dx$. Then based on Fatou's lemma,
\begin{equation*}
\int_{\mathbb R^n}|g_{\varepsilon}(t)|^2\leq\varliminf_{y\in\Gamma,y\to0}\int_{\mathbb R^n}|F(x+iy)|^2dx<\infty.
\end{equation*}
Thus, there exist $g(t)\in L^2(\mathbb R^n)$ and a sequence $\{\varepsilon_k\}$ tending to zero as $k\to\infty$ such that $\lim\limits_{k\to\infty}\int_{\mathbb R^n}g_{\varepsilon_k}(t)h(t)dt=\int_{\mathbb R^n}g(t)h(t)dt$ for $h(t)\in L^2$. In fact, for $t\in -U_{ps}(\Gamma,\psi)$, lemma \ref{supportrelation} implies that $t\in-\Gamma_k^*+\overline{D(0,R_{\psi})}$. Then $t$ can always be written as $t_1+t_2$ with $t_1\in-\Gamma_k^*$ and $|t_2|<R_{\psi}$. Hence, for $y\in\Gamma$,
\begin{equation*}
y\cdot t=y\cdot (t_1+t_2)\leq -|t_1|k+|t_2||y| \leq-(|t|-|t_2|)k+R_{\psi}|t|\leq(R_{\psi}-k)|t|+R_{\psi}k,
\end{equation*}
implying that $\int_{\mathbb R^n}|e^{2\pi y\cdot t}\chi_{-U_{ps}(B_k,\psi)}(t)|^2dt<\infty.$ Therefore, on the right hand side of (\ref{chi}),
\begin{equation*}
\lim_{k\to\infty}\int_{\mathbb R^n}g_{\varepsilon_k}(t)e^{-2\pi iz\cdot t}\chi_{-U_{ps}(\Gamma,\psi)}(t)dt=\int_{\mathbb R^n}g(t)e^{-2\pi iz\cdot t}\chi_{-U_{ps}(\Gamma,\psi)}(t)dt
\end{equation*}
for $e^{2\pi y\cdot t}\chi_{-U_{ps}(\Gamma,\psi)}(t)\in L^2(\mathbb R^n)$. Whilst it is obvious that $F_{\varepsilon}(z)\to F(z)$ when $\varepsilon\to0$. Sending $k$ to $\infty$ on both sides of (\ref{chi}) and letting $f(t)=g(-t)$, we obtain that $f\in L^2(\mathbb R^n)$  and the support supp$f$ is contained in $ U_{ps}(\Gamma,\psi)$, as well as
the desired representation (\ref{mainconclusion}) holds for all $z\in T_\Gamma$.

We now prove the case when $p=\infty$. For $z=(z_1,\ldots,z_n)\in T_{\Gamma}$ and $\varepsilon>0$, we can also construct a function $F_{\varepsilon,y}(x)=F_{\varepsilon}(z)=F_y(x)l^{-1}_{\varepsilon}(z)$, where $l_{\varepsilon}(z)=\left(\prod_{j=1}^n(1-i\varepsilon e_j\cdot z)\right)^{2N}$ with $N>\frac n2$. Then
\begin{equation}
\int_{\mathbb R^n}|F_{\varepsilon,y}(x)|dx\leq \sup_{x\in\mathbb R^n}|F_y(x)|\int_{\mathbb R^n}|l_{\varepsilon}^{-1}(x+iy)|dx\leq \widetilde C_{1,\varepsilon}\|F_y\|_{L^{\infty}(\mathbb R^n)}<\infty\label{FyCtilde}
\end{equation}
and
\begin{equation*}
\int_{\mathbb R^n}|F_{\varepsilon,y}(x)|^2dx\leq\sup_{x\in\mathbb R^n} |F_y(x)|\int_{\mathbb R^n}|l_{\varepsilon}^{-1}(x+iy)|^2dx\leq\widetilde C_{2,\varepsilon}\|F_y\|_{L^{\infty}}(\mathbb R^n)<\infty,
\end{equation*}
where $\widetilde C_{1,\varepsilon}=\int_{\mathbb R^n}\frac{dx}{\left(1+\varepsilon^2B|x|^2\right)^N}$ and $\widetilde C_{2,\varepsilon}=\left(\int_{\mathbb R^n}\frac{dx}{\left(1+\varepsilon^2B|x|^2\right)^{2N}}\right)^{\frac12}<\infty$. Hence $F_{\varepsilon,y}\in L^1(\mathbb R^n)\cap L^2(\mathbb R^n)$. In this case, we also have $g_{\varepsilon}(t)e^{2\pi y\cdot t}=\check F_{\varepsilon,y}(t)\in L^1(\mathbb R^n)$. Then $g_{\varepsilon}(t)e^{2\pi y\cdot t}=\int_{\mathbb R^n}F_{\varepsilon,y}(x)e^{2\pi ix\cdot t}dx$. Therefore, together with (\ref{FyCtilde}),
\begin{eqnarray*}
|g_{\varepsilon}(t)|e^{2\pi (y\cdot t-\psi(y))}&\leq&\widetilde C_{1,\varepsilon}\sup_{x\in\mathbb R^n}|F_y(x)|e^{-2\pi\psi(y)},\\
\sup_{y\in\Gamma}|g_{\varepsilon}(t)|e^{2\pi (y\cdot t-\psi(y))}&\leq&\widetilde C_{1,\varepsilon}\sup_{x\in\mathbb R^n,y\in\Gamma}|F(x+iy)|e^{-2\pi\psi(y)}\\
&=&\widetilde C_{1,\varepsilon}\|F\|_{A^{\infty,\infty}(\Gamma,\psi)}<\infty.
\end{eqnarray*}
Then we can similarly show that supp$g_{\varepsilon}(t)\subset-U_{\infty}(\Gamma,\psi)\subset-\Gamma^*+\overline{D(0,R_{\psi})}$. Applying the same method for $2<p<\infty$, we obtain the desired formula  (\ref{mainconclusion}) holds for all $z\in T_\Gamma$ and  the support supp$f$ is contained in $ U_{\infty}(\Gamma,\psi)\subset \Gamma^*+\overline{D(0,R_{\psi})}$.
\end{proof}

\section{Applications}

In \cite{PD}, denoting by $A_{\alpha}^2(\mathbb C^+)$ a weighted Bergman space of functions holomorphic in $\mathbb C^+$ satisfying $\|F\|^2_{A_{\alpha}^2(\mathbb C^+)}=\int_{\mathbb C^+}|F(x+iy)|^2y^{\alpha}dxdy<\infty$, and by $L_{\beta}^2(\mathbb R^+)$ the space of complex--valued measurable functions $f$ on $\mathbb R^+$ satisfying $\|f\|^2_{L^2_{\beta}(\mathbb R^+)}=\frac{\Gamma(\beta)}{(4\pi)^{\beta}}\int_0^{\infty}|f(t)|^2t^{-\beta}dt<\infty$,
Duren stated an analogy of the Paley--Wiener theorem for Bergman space.

\textbf{Theorem A}(\cite{PD}) For each $\alpha>-1$, the space $A_{\alpha}^2(\mathbb C^+)$ is isometrically isomorphic under the Fourier transform to the space $L^2_{\alpha+1}(\mathbb R^+)$. More precisely, $F\in A_{\alpha}^2(\mathbb C^+)$ if and only if it is the Fourier transform $F(z)=\int_0^{\infty}f(t)e^{2\pi iz\cdot t}dt$ of some function $f\in L^2_{\alpha+1}(\mathbb R^+)$, in which case $\|F\|_{A_{\alpha}^2(\mathbb C^+)}=\|f\|_{L^2_{\alpha+1}(\mathbb R^+)}$.

Based on Theorem 1, letting $s=1$, $p=2$, $\psi (y)=-\frac{\alpha}{4\pi} \log |y|$ and $B$ be a regular open convex cone $\Gamma$, we establish Corollary 2, which can be regarded as a higher dimensional and tube domain generalization of Theorem A.

\begin{corollary}
For each $\alpha>-1$, $F\in A^2_{\alpha}(T_{\Gamma})$ if and only if there exists $f(t)\in L_{\alpha+1}^2(\Gamma^*)$ such that
\begin{equation*}
F(z)=\int_{\Gamma^*}f(t)e^{2\pi iz\cdot t}dt
\end{equation*}
holds for $z\in T_{\Gamma}$ and $\|F\|_{A_{\alpha}^2(T_{\Gamma})}=\|f\|_{L^2_{\alpha+1}(\Gamma^*)}$.
\end{corollary}

\begin{proof}
By restricting the base $B$ to be a regular open convex cone $\Gamma$ and letting $\psi (y)=\psi_{\alpha}(y)=-\frac{\alpha}{4\pi} \log |y|$, $F\in A^2_{\alpha}(T_{\Gamma})$ is also an element of $A^{2,1}(\Gamma,\psi_{\alpha})$. Applying Theorem 1 to such an $F$, we can show that there exists $f(t)$ satisfying (\ref{2g01q}) such that $F(z)=\int_{\mathbb R^n}f(t)e^{2\pi iz\cdot t}dt$ and supp$f\subset U_{1}(\Gamma,\psi_{\alpha})$. Based on (\ref{psicondition}), we have
\begin{equation*}
 R_{\psi_\alpha}=\varlimsup\limits_{y\in \Gamma,y\to\infty} \frac{\psi_\alpha (y)}{|y|}=0.
\end{equation*}
 Thus, together with Lemma \ref{supportrelation}, the supporter  of $f(t)$ is contained  in $\Gamma^*$ and $F(z)=\int_{\Gamma^*}f(t)e^{2\pi iz\cdot t}dt$. Moreover, $\int_{\Gamma}\int_{\Gamma^*}|f(t)|^2e^{-4\pi(y\cdot t+\psi_{\alpha}(y))}dtdy\leq \|F\|_{A^{2,1}(\Gamma,\psi_{\alpha})}$. Thus,
\begin{eqnarray*}
\int_{\Gamma}\int_{\Gamma^*}|f(t)|^2e^{-4\pi(y\cdot t+\psi_{\alpha}(y))}dtdy=\int_{\Gamma^*}\int_{\Gamma}|f(t)|^2e^{-4\pi y\cdot t}y^{\alpha}dydt=\int_{\Gamma^*}|f(t)|^2\frac{\Gamma(\alpha)}{(4\pi t)^{\alpha+1}}dt,
\end{eqnarray*}
which shows $f\in L_{\alpha+1}^2(\Gamma^*)$. And Plancherel's Theorem assures that $\|F\|_{A_{\alpha}^2(T_{\Gamma})}=\|f\|_{L^2_{\alpha+1}(\Gamma^*)}$.

Conversely, note that $F(z)=\int_{\Gamma^*}f(t)e^{ 2\pi it\cdot z}dt$. For $f(t)\in L_{\alpha+1}^2(\Gamma^*)$, Plancherel's theorem implies that
\begin{eqnarray*}
\int_{\mathbb R^n}|F(x+iy)|^2dx&=&\int_{\Gamma^*}e^{-4\pi y\cdot t}|f(t)|^2dt,\\
\int_{\Gamma}\int_{\mathbb R^n}|F(x+iy)|^2e^{-4\pi\psi_{\alpha}(y)}dxdy&=&\int_{\Gamma}\int_{\Gamma^*}|f(t)|^2e^{-4\pi(y\cdot t+\psi_{\alpha}(y))}dtdy<\infty,
\end{eqnarray*}
in which $\psi_{\alpha}(y)=-\frac{\alpha}{4\pi} \log |y|$. Hence, $F(z)\in A^{2,1}(\Gamma,\psi_{\alpha})=A^2_{\alpha}(T_{\Gamma})$. The proof is complete.
\end{proof}

By restricting the base $B$ to be a regular open convex cone $\Gamma$, we establish the following weighted version of the edge-of-the-wedge theorem (see \cite{WR2}) as an application of Theorem 1.
\begin{theorem}
Assume that $\Gamma$ is a regular open convex cone in $\mathbb R^n$ ,  $\psi_1\in C(\Gamma)$ and   $\psi_2\in C(-\Gamma)$ satisfy
\begin{equation}
 R_{\psi_1}=\varlimsup\limits_{y\in \Gamma,y\to\infty} \frac{\psi_1 (y)}{|y|}<\infty \label{4condition-1}
\end{equation}
and
\begin{equation}
 R_{\psi_2}=\varlimsup\limits_{y\in \Gamma,y\to\infty} \frac{\psi_2 (-y)}{|y|}<\infty \label{4condition-2}
\end{equation}
respectively.
 If $1< p\leq2$, $0<s(p-1)\leq 1$ , $F_1\in A^{p,s}(\Gamma,\psi_1)$ and $F_2\in A^{p,s}(-\Gamma,\psi_2)$, satisfying
\begin{equation}
\varliminf\limits_{y\to0}\int_{\mathbb R^n}|F_1(x+iy)-F_2(x-iy)|^pdx=0,\label{4conditionintegrable}
\end{equation}
then $F_1$ and $F_2$ can be analytically extended to each other and further form an entire function $F$. Furthermore, there exists a  function $f\in L^1(\mathbb{R}^n)$ supported in a bounded convex set $K$ such that $F(z)=\int_K f(t)e^{2\pi it\cdot z}dt$.
\end{theorem}

\begin{proof}
Theorem 1 implies that there exists a function $f_j$$(j=1,2)$  such that
$$
F_j=\int_{\mathbb R^n}f_j(t)e^{2\pi it\cdot z}dt
$$ holds, in which the supporter of $f_j$ is contained  in $U_{sp}((-1)^{j+1}\Gamma,\psi_j)$ for  for $1<p\leq2$. Based on lemma \ref{supportrelation}, supp$f_j\subset (-1)^{j+1}\Gamma^*+\overline{D(0,R_{\psi_j})}$. By the Hausdorff-Young inequality,
\begin{equation*}
\left(\int_{\mathbb R^n}|f_1(t)e^{2\pi y\cdot t}-f_2(t)e^{-2\pi y\cdot t}|^qdt\right)^{\frac1q}\leq \left(\int_{\mathbb R^n}|F_1(x+iy)-F_2(x-iy)|^pdt\right)^{\frac1p}.
\end{equation*}
Then it follows from Fatou's lemma and (\ref{4conditionintegrable}) that $\|f_1- f_2\|_{L^q(\mathbb R^n)}=0$. Thus, $f_1=f_2$ almost everywhere on $\mathbb R^n$. Let $f_1(t)=f_2(t)=f(t)$, and $R=\max\{R_{\psi_1}, R_{\psi_2}\}$, then supp$f\subset K\subset(\Gamma^*+\overline{D(0,R )})\bigcap(-\Gamma^*+\overline{D(0,R )})$. Thus, $K$ is a bounded convex set. Consequently, $F(z)=\int_Ke^{2\pi iz\cdot t}f(t)dt$ is an entire function, where $F(z)=F_1(z)$ for $z\in T_{\Gamma}$ and $F(z)=F_2(z)$ for $z\in T_{-\Gamma}$.
\end{proof}

Similarly, we can prove the weighted version of the edge-of-the-wedge theorem for $p>2$.
\begin{theorem}
Suppose that $\Gamma$ is a regular open convex cone in $\mathbb R^n$,$\psi_1\in C(\Gamma)$ and   $\psi_2\in C(-\Gamma)$ satisfy (\ref{4condition-1}) and (\ref{4condition-2})
respectively.
 If $F_1\in A^{p,s}(\Gamma,\psi_1)$ and $F_2\in A^{p,s}(-\Gamma,\psi_2)$, where $p>2$, satisfying
\begin{equation}
\varliminf_{y\in\Gamma,y\to0}\int_{\mathbb R^n}|F_1(x+iy)-F_2(x-iy)|^2dx=0,\label{conditionintegrable}
\end{equation}
then $F_1$ and $F_2$ can be analytically extended to each other and further form an entire function $F$. Furthermore, there exists a measurable function $f(t)$ supported in a bounded convex set $K$ such that $F(z)=\int_K f(t)e^{2\pi it\cdot z}dt$.
\end{theorem}
\begin{proof}
For $F_j\in A^{p,s}((-1)^{j+1}\Gamma,\psi_j)$ and $\frac1p+\frac1q=1$, exists a measurable function $f_j$ such that $F_j=\int_{\mathbb R^n}f_j(t)e^{2\pi it\cdot z}dt$ and supp$f_j\subset U_{sp}((-1)^{j+1}\Gamma,\psi_j)$, where $j=1,2$. It then follows from Lemma \ref{supportrelation} that supp$f_j \subset (-1)^{j+1}\Gamma^*+\overline{D(0,R_{\psi_j})}$. Plancherel's Theorem implies that
\begin{equation*}
\left(\int_{\mathbb R^n}|f_1(t)e^{2\pi y\cdot t}-f_2(t)e^{-2\pi y\cdot t}|^2dt\right)^{\frac12}=\left(\int_{\mathbb R^n}|F_1(x+iy)-F_2(x-iy)|^2dx\right)^{\frac12}.
\end{equation*}
Then based on (\ref{conditionintegrable}) and Fatou's Lemma, $\|f_1-f_2\|_{L^2(\mathbb R^n)}=0$, which means $f_1=f_2$ almost everywhere on $\mathbb R^n$. Let $f_1(t)=f_2(t)=f(t)$ and $R=\max\{R_{\psi_1}, R_{\psi_2}\}$, then supp$f(t)\subset K=(\Gamma^*+\overline{D(0,R )})\bigcap(-\Gamma^*+\overline{D(0,R )})$. Thus, $K$ is a bounded convex set. As a result, $F(z)=\int_Ke^{2\pi iz\cdot t}f(t)dt$ is an entire function, where $F(z)=F_1(z)$ for $z\in T_{\Gamma}$ and $F(z)=F_2(z)$ for $z\in T_{-\Gamma}$.
\end{proof}

%%%%%%%%%%%%%%%%%%%%%%%%%%%%%%%%%%%%%%%%%%%%%%%%%%%%%%%%%%%%%%%%%%%%%%%%%%%%%%%%%%%%%%%%%%%%%%%%%%
%%%%%%%%%%%%%%%%%%%%%%%%%%%%%%%%%%%%%%%%%%%%%%%%%%%%%%%%%%%%%%%%%%%%%%%%%%%%%%%%%%%%%%%%%%%%%%%%%%
%%%%%%%%%%%%%%%%%%%%%%%%%%%%%%%%%%%%%%%%%%%%%%%%%%%%%%%%%%%%%%%%%%%%%%%%%%%%%%%%%%%%%%%%%%%%%%%%%%

%%%%%%%%%%%%%%%%%%%%%%%%%%%%%%%%%%%%%%%%%%%%%%%%%%%%%%%%%%%%%%%%%%%%%%%%%%%%%%%%%%%%%%%%%%%%%%%%%%
%%%%%%%%%%%%%%%%%%%%%%%%%%%%%%%%%%%%%%%%%%%%%%%%%%%%%%%%%%%%%%%%%%%%%%%%%%%%%%%%%%%%%%%%%%%%%%%%%%
%%%%%%%%%%%%%%%%%%%%%%%%%%%%%%%%%%%%%%%%%%%%%%%%%%%%%%%%%%%%%%%%%%%%%%%%%%%%%%%%%%%%%%%%%%%%%%%%%%
%%%%%%%%%%%%%%%%%%%%%%%%%%%%%%%%%%%%%%%%%%%%%%%%%%%%%%%%%%%%%%%%%%%%%%%%%%%%%%%%%%%%%%%%%%%%%%%%%%

\vskip 0.5cm{

\end{document}